\documentclass[12pt]{amsart}  
\usepackage{amssymb,amsmath,amsthm,amscd}
\usepackage[all]{xy}

\usepackage{graphicx}

\usepackage{hyperref}

\numberwithin{equation}{section}

\newtheoremstyle{fancy1}{10pt}{10pt}{\itshape}{12pt}{\textsc\bgroup}{.\egroup}{8pt}{
}
\newtheoremstyle{fancy2}{10pt}{10pt}{}{12pt}{\itshape}{.}{8pt}{ }

\theoremstyle{fancy1}

\newtheorem{cor}[equation]{Corollary}
\newtheorem{lem}[equation]{Lemma}
\newtheorem{prop}[equation]{Proposition}

\newtheorem{main}{Theorem}
 
\newtheorem*{main*}{Theorem}

\newtheorem*{cor*}{Corollary}
 
 \newtheorem*{thm*}{Theorem}

\newtheorem*{problem*}{Problem}

\theoremstyle{fancy2}
\newtheorem{definition}[equation]{Definition}
\newtheorem{rem}[equation]{Remark}
\newtheorem*{rem*}{Remark}

\newtheorem*{rems*}{Remarks}

\newcommand{\cref}[1]{Corollary~\ref{#1}}
\newcommand{\dref}[1]{Definition~\ref{#1}}

\newcommand{\lref}[1]{Lemma~\ref{#1}}
\newcommand{\pref}[1]{Proposition~\ref{#1}}
\newcommand{\rref}[1]{Remark~\ref{#1}}
\newcommand{\tref}[1]{Theorem~\ref{#1}}
\newcommand{\sref}[1]{Section~\ref{#1}}

\newcommand{\e}{\epsilon}
\newcommand{\ii}{isometric immersion }

\newcommand{\RP}{\mathbb{R\mkern1mu P}}
\newcommand{\CP}{\mathbb{C\mkern1mu P}}

\newcommand{\Sph}{\mathbb{S}}

\newcommand{\R}{{\mathbb{R}}}
\newcommand{\Z}{{\mathbb{Z}}}

\newcommand{\fk}{{\mathfrak{k}}}

\newcommand{\fso}{{\mathfrak{so}}}

\newcommand{\fu}{{\mathfrak{u}}}

\def\con#1=#2(#3){#1 \equiv #2 \bmod{#3}}

\newcommand{\ml}{\langle}                     
\newcommand{\mr}{\rangle}
\newcommand{\la}{\langle}
\newcommand{\ra}{\rangle}
\newcommand{\pepe}{wide}                   

\newcommand{\diag}{\ensuremath{\operatorname{diag}}}

\newcommand{\rank}{\ensuremath{\operatorname{rk}}}

\newcommand{\im}{\ensuremath{\operatorname{Im}}}

 \DeclareMathOperator{\Iso}{Iso}

\DeclareMathOperator{\Id}{Id}
 
\newcommand{\spa}{\mbox{span}}
\newcommand{\rk}{\ensuremath{\operatorname{rank}}}

\newcommand{\abs}[1]{\left\vert#1\right\vert}

\begin{document}
\title{Nonnegatively curved Euclidean submanifolds\\ in codimension two}  

\author{Luis A. Florit}
\address{IMPA: Est. Dona Castorina 110, 22460-320, Rio de Janeiro,
Brazil}
\email{luis@impa.br}
\author{Wolfgang Ziller}
\address{University of Pennsylvania: Philadelphia, PA 19104, USA}
\email{wziller@math.upenn.edu}
\thanks{The first author was supported by CNPq-Brazil,
and the second author  by  a grant from the National Science
Foundation, by IMPA, as well as CAPES.}

\begin{abstract}  
\bigskip
We provide a classification of compact Euclidean submanifolds
$M^n\subset\R^{n+2}$ with nonnegative sectional curvature, for $n\ge 3$.
The classification is in terms of the induced metric (including the
diffeomorphism classification of the manifold), and we study the
structure of the immersions as well. In particular, we provide the
first known example of a nonorientable quotient
$(\Sph^{n-1}\times\Sph^1)/{\Z_2}\subset\R^{n+2}$ with nonnegative
curvature. For the 3-dimensional case, we show that either the universal
cover is isometric to $\Sph^2\times\R$, or $M^3$ is diffeomorphic to a
lens space, and the complement of the (nonempty) set of flat points is
isometric to a twisted cylinder $(N^2\times\R)/\Z$. As a consequence we
conclude that, if the set of flat points is not too big, there exists a
unique flat totally geodesic surface in $ M^3$ whose complement is
the union of one or two twisted cylinders over disks.
\end{abstract}
\maketitle
\bigskip

It is well known that a compact immersed positively curved hypersurface
in Euclidean space is diffeomorphic to a sphere (and is in fact the
boundary of a convex body), the diffeomorphism simply given by the Gauss
map. A deeper result, that such a submanifold still has to be
homotopy equivalent to a sphere in codimension 2, is due to A. Weinstein
\cite{We} and D. Moore \cite{m2}. However, the situation where the
sectional curvature of the submanifold is only nonnegative is more
delicate. Although it is still not known whether there exists an
isometric immersion of $\RP^2$ into $\R^4$ with nonnegative curvature,
the higher dimensional problem was studied in \cite{bm1,bm2}. Using
recent results about the Ricci flow, we strengthen these
by showing the following.

\begin{main}\label{a}
Let $f:M^n\to\R^{n+2}$,  $n\geq 3$, be an \ii of a compact
Riemannian manifold with nonnegative sectional curvature.
Then one of the following holds:
\begin{itemize}
\item[$a)$] $M^n$ is diffeomorphic to $\Sph^n$;
\item[$b)$] $M^n$ is isometric to a product metric on
$\Sph^k\times\Sph^{n-k}$ for some $2\le k\le n-2$, and $f$ is the
product embedding of two convex Euclidean hypersurfaces;
\item[$c)$] $M^n$ is isometric to $(\Sph^{n-1}\times\R)/\Gamma$ with a
product metric on $\Sph^{n-1}\times\R$ and $\Gamma\simeq\Z$ acting
isometrically. As a manifold, $M^n$ is diffeomorphic to
$\Sph^{n-1}\times\Sph^1$ if orientable, or diffeomorphic to the
nonorientable quotient $(\Sph^{n-1}\times\Sph^1)/\Delta\Z_2$ otherwise,
where $\Delta\Z_2$ denotes some diagonal $\Z_2$ action;
\item[$d)$] $M^n$ is diffeomorphic to a 3 dimensional
lens space $\Sph^3/\Z_k$.
\end{itemize}
\end{main}

Not only is this statement stronger than previously known results, but
thanks to the use of the Ricci flow its proof becomes substantially
simpler as well. More importantly, in the process we will also provide
new strong restrictions on the structure of these submanifolds. In fact,
the particular nature of the submanifolds in case $(c)$, and even if a
nonorientable immersion exists, was unsettled in the literature.
Furthermore, if case $(d)$ is possible at all was never discussed. The
main purpose of our paper is to address both issues.

\smallskip

Regarding case $(a)$, there are plenty of immersions of $\Sph^n$ with
nonnegative curvature in codimension two. For example, take any compact
convex hypersurface and then a composition with a flat (not necessarily
complete) hypersurface. On the other hand, the submanifolds in case $(b)$
are isometrically rigid.

\bigskip

It was the study of the structure of the immersions that led us to the
following non-orientable example in case $(c)$.

\medskip

{\it Example 1\label{ex1}: The nonorientable quotient
$(\Sph^{n-1}\times\Sph^1)/\Delta\Z_2$ embedded in $\R^{n+2}$.}

\medskip

Consider a flat strip  isometrically immersed in $\R^3$
$$\beta:\R\times(-\epsilon,\epsilon)\to\R^3$$
such that
$ \beta(x_0+1,-x_1) = \beta(x_0,x_1).$
The image is a flat Moebius band immersed or embedded in $\R^3$. Large
families of analytic Moebius bands of this type, together with some
classification results, have been given in \cite{ck,sc,wu}.

The product immersion of $\beta$ with the identity gives
a flat hypersurface
$$
h=\beta\times\Id_{\R^{n-1}}:\R\times(-\epsilon,\epsilon)\times\R^{n-1}
\to\R^3\times\R^{n-1}=\R^{n+2}.
$$
Using a convex hypersurface
$g\colon N^{n-1}\cong\Sph^{n-1}\to (-\epsilon,\epsilon)\times\R^{n-1}$
invariant under the reflection in the first coordinate of $\R^{n-1}$,
we define the cylinder $\Id_\R\times g\colon\R\times
N^{n-1}\to\R\times(-\epsilon,\epsilon)\times\R^{n-1}$.
The composition
$$
f=h\circ(\Id_\R\times g)\colon \R\times N^{n-1}\to \R^{n+2}
$$
then satisfies $f\circ\tau=f$, where $\tau$ is the map
$\tau(x_0,x_1,x_2,\dots,x_n)=(x_0+1,-x_1,x_2,\dots,x_n)$.
Hence the image of $f$ is isometric to the {\it twisted cylinder}
$(\R\times N^{n-1})/\Z$, where $\Z$ is generated by the orientation
reversing isometry of $N^{n-1}$ induced by $\tau$. Thus $h\circ f$
descends to the desired immersion
$$
f':(\Sph^1\times N^{n-1})/\Delta\Z_2\simeq (\R\times N^{n-1})/\Z\to\R^{n+2}.
$$
Observe that we can also choose $\beta$ as a cylinder and obtain
immersions of $(N^{n-1}\times\R)/\Z\simeq \Sph^{n-1}\times\Sph^1$ which
are not products of immersions as in part $(b)$. Actually, they are not
even locally product of immersions, as is the case in Example 2 below.

\medskip

We say that an isometric immersion $f$ is a {\it composition (of $j$)}
when $f=h\circ j$, where $j:M^n\to N_0^{n+1}$ is an isometric immersion
into a (not necessarily complete) flat Euclidean hypersurface
$h\colon N_0^{n+1}\to \R^{n+2}$. We will see that all submanifolds in
\tref{a} $(c)$ are almost everywhere compositions of a cylinder over a
strictly convex Euclidean hypersurface when the Ricci curvature is
2-positive. For complete simply connected nowhere flat nonnegatively
curved manifolds which split off a line, this structure of a composition
was shown in \cite{bdt} under certain regularity assumptions. It is not
clear though what additional restrictions hold, or whether regularity is
necessary, if the immersion $f$ descends to a compact quotient, as
required in \mbox{\tref{a} $(c)$}.

\medskip

The following example, where the induced intrinsic metric is well known
in the theory of graph manifolds, illustrates the structure necessary for
case $(d)$ to occur.

\medskip

{\it Example 2: The switched\, $\Sph^3$ in $\R^5$.}

\medskip

Consider a closed strictly convex hemisphere inside a closed halfspace,
$$
\Sph^2_+\subset \R^3_+=\R^2\times \R_+,
$$
such that the boundary is a closed geodesic in
$\Sph^2_+$ along which the Gauss curvature vanishes to infinite order, and
with image contained in $\R^2\times\{0\}$. Its product with
another $\Sph^1\subset\R^2$ gives the nonnegatively curved three manifold
$$
N^3_+=\Sph^2_+\times\Sph^1\subset\R^5_+=\R^2\times\R_{\geq 0}\times \R^2,
$$
whose boundary is the totally geodesic flat torus
$$
T^2=\Sph^1\times\Sph^1\subset \R^2\times\{0\}\times\R^2=\R^4\subset\R^5.
$$
Now, reversing the role of the $\R^2$ factors we similarly construct
another manifold
$N^3_-=\Sph^1\times\Sph^2_-\subset \R^5_-=\R^2\times \R_{\leq 0}\times
\R^2$ with the same boundary $T^2\subset\R^4$. Thus
$$ M^3=N^3_+\cup N^3_-\subset\R^5$$
gives an embedded Euclidean submanifold with nonnegative
sectional curvature:

\begin{figure}[!ht]
\centering
\includegraphics[width=0.3\textwidth,natwidth=633,natheight=540]{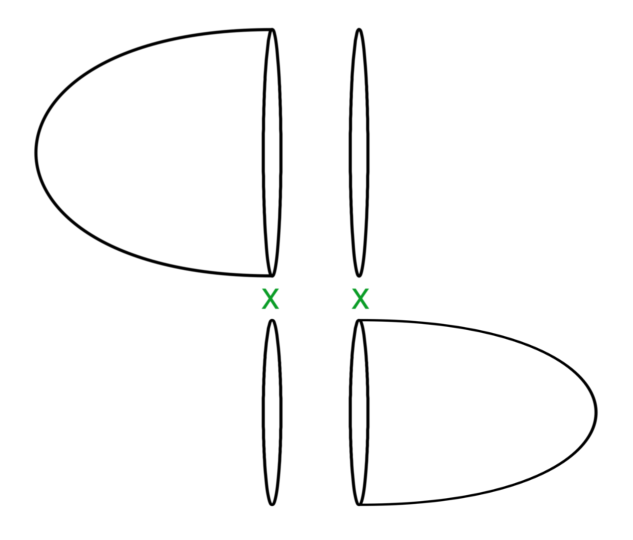}
\centerline{\small A switched $\Sph^3$ isometrically embedded in $\R^5$}
\end{figure}

As a smooth manifold $M^3$ is diffeomorphic to the 3-sphere since this is
the usual description of $\ \Sph^3$ as the union of 2 solid tori. The set
of flat points of $M^3$ contains the totally geodesic flat torus
$T^2\subset M^3$ which disconnects $M^3$. Moreover, each connected
component of $M^3\setminus T^2$ is obtained as a composition of the
cylinder $\R\times \Sph^2_+\subset\R^4$ with a local \ii from $\R^4$ into
$\R^5$. By replacing the torus $T^2$ with a small flat cylinder
$C_\epsilon=\Sph^1\times[0,\epsilon]\times\Sph^1\subset \R^5$, we obtain
$$
\Sph^3\cong M^3_\epsilon=N^3_+\cup C_\e \cup N^3_-\subset\R^5.
$$

The resulting embeddings $M^3_\e\subset\R^5$ for $\e\geq 0$
have thus the following properties:
\newpage
\medskip
\begin{itemize}
\item $M^3_\epsilon$ is compact with nonnegative sectional curvature;
\item The set of flat points contains a totally geodesic flat torus
that disconnects $M^3_\epsilon$;
\item $M^3_\epsilon$ is locally an isometric product of two manifolds
for $\e>0$ (almost everywhere for $\e=0$);
\item The embedding $M^3_\epsilon\subset\R^5$ is locally, yet not
globally, a composition of a cylindrical hypersurface for $\e>0$ (almost
everywhere for $\e=0$);
\item $M^3_\epsilon$ has no points with positive Ricci curvature;
\item $M^3_\epsilon$ has 2-positive Ricci curvature outside the set of
flat points.
\end{itemize}
\medskip
The submanifold $M^3_\epsilon$ can be changed further so that the
boundary of the set of flat points is not regular, and arbitrarily
complicated sets of flat points can be introduced in $\Sph^2_\pm$ as
well. Observe that, by taking the two disks $\Sph^2_\pm$ with strictly
convex boundaries as plane curves, all these properties can be achieved
even with 1-regular full embeddings, i.e., with the dimension of the
space spanned by the second fundamental form having constant dimension 2
everywhere.
\smallskip

The following result recovers some of this structure for any
immersion of a lens space with nonnegative sectional curvature.

\begin{main}\label{3a}
Let $f:M^3=\Sph^3/\Z_k \to \R^5$ with $k>1$ be an isometric immersion with
nonnegative sectional curvature. Then the set $M_0$ of flat points of
$M^3$ is nonempty, and $M^3\setminus M_0$ is isometric to a twisted
cylinder $(N^2\times\R)/\Z$, where $N^2$ is a (not necessarily connected)
surface with positive Gaussian curvature. Moreover, $f$ is a composition
almost everywhere on $M^3\setminus M_0$.
\end{main}

\smallskip

The natural intrinsic problem raised by this result is to try to
understand how twisted cylinders can be glued together with flat regions
in order to build a compact manifold. One way is to glue the cylinders
through compact totally geodesic hypersurfaces, as we did for $M^3_\e$ in
Example 2 above.
In \cite{fz} we study this intrinsic structure 
under the additional assumption that
the set of nonflat points $M\setminus M_0$ is dense and locally finite (without assuming that the curvature is necessarily nonnegative). The main result in \cite{fz} is that under this assumption there
exist pairwise disjoint flat totally geodesic compact surfaces in
$M^3$ whose complement is a union of twisted cylinders. We also give examples of metrics on $T^3$ which show that the assumption on the nonflat points is necessary. In the case of nonnegative curvature, Theorem B above, together with Theorem C in \cite{fz}, implies that:

\begin{cor*}\label{main2}
Let $f:M^3=\Sph^3/\Z_k \to \R^5$, $k>1$, be an isometric immersion with
nonnegative sectional curvature, such that $M^3\setminus M_0$ is dense
and its number of connected components is locally finite. Then, either
$M_0$ contains a unique flat totally geodesic torus $T^2$ that
disconnects $M^3$ and $M^3=V_1\sqcup T^2 \sqcup V_2$, or $M^3$ is the
lens space $L(4p,2p+1)$ and $M_0$ contains a unique flat totally geodesic
 Klein bottle $K^2$ satisfying $M^3=V_1\sqcup K^2$. In both
cases, $V_i=(D_i^2\times\R)/\Z$ is a twisted cylinder, where $D_i^2$ is a
disc with nonnegative Gaussian curvature and boundary a closed geodesic.
\end{cor*}

 Yet, despite the rigid structure described
above, we still do not know if a nonnegatively curved lens space in
$\R^5$ actually exists.

\medskip

For the proof of \tref{a} we will need a fundamental property of our
submanifolds that was proved in \cite{Ch} and \cite{m2}. In order to state
it, recall that we have the {\it type numbers} $\tau_k$ of the immersion:
if $h_v=\ml v,\cdot\,\mr:M^n\to\R$ is a height function, and $\mu_k(v)$
the number of critical points of $h_v$ with index $k$, then $\tau_k$ is
the average of $\mu_k(v)$, integrated over the unit sphere in $\R^{n+2}$.
If $M^n$ is compact with nonnegative sectional curvature immersed in
$\R^{n+2}$, then the type numbers satisfy
\begin{equation*}
\tau_0+\tau_n\ge \tau_1 + \dots + \tau_{n-1}.
\end{equation*}
Furthermore, if the inequality is strict, we will see in \sref{RG}
that $M^n$ is diffeomorphic to a sphere. We are thus mainly interested in
the case of equality, and call such immersions {\it wide}. In the
situation of \tref{3a}, we will show that
$\tau_k=\frac{1}{8\pi^2}\int_{N^2}K(x)\kappa_g(x)dx$ for $0\leq k\leq3$,
where $\kappa_g(x)$ denotes the total curvature of the image of the $\R$
factor through $x\in N^2$, and $K$ the Gaussian curvature~of~$N^2$.

For codimension two Euclidean submanifolds with nonnegative curvature,
A. Weinstein \cite{We} showed that at every point there exists a basis
$\{\xi,\eta\}$ of the normal space such that its shape operators $A$ and
$B$ are nonnegative, and hence the curvature operator is nonnegative. The
crucial new property that we will prove is that for a wide immersion we
have in addition that
\begin{equation*}
\text{either} \ \ker A\cap\ker B \neq 0, \ \ \text{or} \
\ \ker A\oplus \ker B = T_p M\ \, \text{with} \ A,B\neq 0,
\end{equation*}

\medskip

\noindent and call the (not necessarily complete) immersions which
satisfy only this property {\it locally wide}. These immersions already
have strong rigidity properties. We examine their behavior on the open
subsets
$U_k = \{p\in M^n: \rk\, A(p) = n-k,\,\rk\, B(p)
= k \ \ \text{and} \ \ker A(p)\cap \ker B(p) = 0\}$,
which fill out the complement of the points of positive relative nullity.
We will show in particular that:

\medskip

\begin{itemize}
\item If $2\leq k \leq [n/2]$, then $f|_{U_k}$ is locally a product
immersion of two strictly convex Euclidean hypersurfaces;
\item If $\pi:\tilde U_1\to U_1$ is the universal cover, then
$(f\circ\pi)$ is a composition of a convex Euclidean hypersurface with
constant index of relative nullity one (see \tref{b1});
\item If $M^n$ is complete and has no points of positive Ricci curvature,
but has 2-positive Ricci curvature, then $\tilde M^n=\Sph^{n-1}\times\R$,
where $g:\Sph^{n-1}\to\R^n$ is a compact strictly convex hypersurface.
\end{itemize}

\section{Preliminaries}

Let $f:M^n\to\R^{n+2}$, $n\geq 3$, be an \ii of a Riemannian manifold
with nonnegative sectional curvature. A. Weinstein observed
that, in this case, the curvature operator $\hat R$ is nonnegative.
Indeed, if $\alpha$ is the second fundamental form, the Gauss equations
imply that $\ml \alpha(X,X),\alpha(Y,Y)\mr \ge 0$ at every $p\in M$.
Hence, in the connected subset $\{\alpha(X,X):X\in T_pM\}\subset
T^\perp_p M$ intersected with the unit circle all points have distance at
most $\pi/2$. This implies that this set lies in the first quadrant with
respect to some orthonormal basis $\{\xi,\eta\}$, while it lies in its
interior if it has positive sectional curvature. We thus have the shape
operators
\begin{equation}\label{AB}
A=A_{\xi} {\rm \ \ \ and \ \ \ } B=A_{\eta}, \ \ \
{\rm \ \ \ with\ \ \ } A\ge 0 {\rm \ \ \ and \ \ \ } B\ge 0.
\end{equation}
since $\ml A_\tau(X),Y\mr=\ml\alpha(X,Y),\tau\mr$.

By the Gauss equations we have $\hat{R}=\Lambda^2A+\Lambda^2B$, and hence
$\hat{R}\ge 0$ (and $\hat{R}>0$ if the sectional curvature is positive).
Furthermore, the cone of shape operators which are positive semidefinite
contains the first quadrant, while the cone of shape operators which are
negative semidefinite contains the third.

We will see that the interesting immersions are quite rigid as a
consequence of the following purely algebraic key result which will be
applied to the shape operators in \eqref{AB}.

\begin{lem}\label{l:eq}
Let $A$ and $B$ be nonnegative semidefinite self--adjoint operators
on an Euclidean space $V$. Then,
$|\det\left(A+tB\right)|\geq|\det\left(A-tB\right)|,\ \ \forall\, t>0.$
Moreover, equality holds for all $t>0$ if and only if either
$\ker A \cap \ker B \neq 0$, or $\ker A \oplus \ker B=V$.
\end{lem}
\proof First, observe that the lemma follows easily if one of the
operators is nonsingular. Indeed, if, say, $B$ is invertible,
$\det\left(A+tB\right) = \det(B)\phi(t)$, where
$\phi$ is the characteristic polynomial of the positive semidefinite
operator $B^{-1/2}AB^{-1/2}$. Since all the coefficients of $\phi$ are
nonnegative the inequality holds, with equality if and only if
$A=0$. Therefore, we can assume that both $A$ and $B$ are singular.

Let $K=\ker B$, and decompose $A$ as its block endomorphisms
$(C,D,D^t,E)$, according to the orthogonal decomposition
$K\oplus K^\perp$,
$$
A(u,v) = (Cu+Dv,D^tu+Ev), \ \ \ \forall u\in K, v\in K^\perp.
$$
Since $A\geq 0$, we have $C\geq0$ and $E\geq 0$.

For $u\in\ker C$ we get
$0\leq\la A(su,v), (su,v)\ra = 2s\la D^tu,v\ra+\la Ev,v\ra$,
for all $s\in\R$. Hence, $D^tu=0$ and $u \in K \cap \ker A$. Therefore
either $C>0$ or the kernels intersect and equality holds.
So let us then assume further that $C>0$ and $K\cap \ker A = 0$.

If we call $\hat B$ the symmetric operator $B$ restricted and projected
to $K^\perp$, then $\hat B>0$ and, for all $u\in K, v\in K^\perp$,
$$(A+tB)(u,v) = (Cu+Dv,D^tu+(E+t\hat B)v).$$
Thus,
$$\det(A+tB)=\det(C)\det\left(E+t\hat B-D^tC^{-1}D\right)
= \det(C)\det(\hat A+t\hat B),$$
where $\hat A$ is the symmetric operator on $K^\perp$ given by
$\hat A=E-D^tC^{-1}D$. Now, observe that~$\hat A$ is nonnegative since
$$\la(E-D^tC^{-1}D)v,v\ra = \la Ev,v\ra - \la C^{-1}Dv,Dv\ra
= \la A(-C^{-1}Dv,v),(-C^{-1}Dv,v)\ra \geq 0.$$
This gives the inequality in the lemma as in the nonsingular case.
Moreover, this computation also shows that $v\in\ker \hat A$ if and only
if $(-C^{-1}Dv,v)\in\ker A$. Thus, $\det(A+tB)$ is again, up to a
positive constant, the characteristic polynomial of the positive
semidefinite operator $\hat B^{-1/2}\hat A\hat B^{-1/2}$, which is an odd
or even function if and only if $\hat A=0$. In this case,
$\dim \ker A = \dim K^\perp$, as claimed.

The converse of the equality case is straightforward.
\qed

\medskip

Since  $\hat{R}\ge 0$ we can
use the following rigidity result, where $hol$ denotes the Lie algebra of
the holonomy group of $M$ (see Theorem 1.13 in \cite{Wi}):

\begin{prop}\label{rigidity}
If $M$ is a compact and simply connected Riemannian
manifold with $\hat{R}\ge 0$, then one of the following holds:
\begin{itemize}
\item $hol=\fso(n)$ and $M$ is diffeomorphic to $\Sph^n$, with the Ricci
flow converging to a metric with constant positive curvature;
\item $hol=\fu(n)$ and $M$ is diffeomorphic to $\CP^n$, with the Ricci
flow converging to the Fubini Study metric;
\item $M$ is isometric to an irreducible symmetric space $G/K$ with
$hol=\fk$;
\item Arbitrary isometric products of the above cases.
\end{itemize}
\end{prop}

Notice in particular that if $hol=\fso(n)$ in the above, then $M$ is
diffeomorphic to a sphere since, if the Lie algebra $\fk$ of $K$ for a
symmetric space $G/K$ satisfies $\fk=\fso(n)$, then $G/K$ is in fact
isometric to a round $n$-sphere.

\medskip

We will also need the next result due to Bishop (\cite{Bi}) about the
holonomy group of a general compact Euclidean submanifold in codimension
two.
\begin{prop}\label{bishop}
If $M^n$ is a compact submanifold immersed in $\R^{n+2}$, then either
$n=4$ and $hol=\mathfrak{u}(2)$, or $hol=\fso(k)\times\fso(n-k)$ for some
$0\leq k \leq [n/2]$.
\end{prop}
Here $k=0$ corresponds to $hol=\fso(n)$ and $k=1$ to $hol=\fso(n-1)$,
i.e., the manifold splits a flat factor locally. In particular, there can
be at most a one dimensional flat factor.

\smallskip

\section{Rigidity}\label{RG}  

\medskip

In this section we provide the proof of \tref{a}. In the process, we
reprove, simplify and strengthen known results about compact Euclidean
submanifolds in codimension 2 with nonnegative sectional curvature by
using \pref{rigidity}. Along the way, we will
see how the proof implies further rigidity of the immersion.

\bigskip
Let $f$ and $M^n$ be as in \tref{a}. Since $M^n$ is compact with
nonnegative sectional curvature, a well-known consequence of the
Cheeger--Gromoll splitting  theorem, see  \cite{CG} Theorem C,
implies that its universal cover
$\tilde{M}^n$ splits isometrically as
\begin{equation}\label{soul}
\tilde{M}^n=\R^\ell\times N^{n-\ell},
\end{equation}
where $N^{n-\ell}$ is compact and simply connected.
Since the Lie algebra of the holonomy group of $M$ and $\tilde M$
coincide, \pref{bishop} implies that $\ell\le 1$.

Now observe that in our context $n=4$ and
$hol=u(2)$ is not possible in \pref{bishop} since then the universal
cover would be compact, and by \pref{rigidity} diffeomorphic to $\CP^2$.
But $\CP^2$ does not admit an immersion in $\R^6$ as follows, e.g.,
from $p_1(\CP^2)=3$, the product formula for Pontrjagin classes, and
$p_1(E)=e(E)^2$ for the normal 2 plane bundle~$E$ of the immersion.

Therefore, the existence of an \ii implies that
\begin{equation}\label{k}
hol=\fso(k)\oplus\fso(n-k),
\end{equation}
for\ some $0\leq k \leq [n/2]$.
We can now combine \eqref{soul}, \eqref{k} and \pref{rigidity}.
First, it follows that $\ell\le 1$ and $\ell=1$ if and only if $k=1$.

If $\ell=k=0$, then $\tilde{M}^n$ is compact and hence diffeomorphic to
$\Sph^n$, while if $\ell=0$ and $k\ge 2$, then $\tilde{M}^n$ is
diffeomorphic to $\Sph^k\times\Sph^{n-k}$. In the latter case, by \cite{am} the immersion must be a product immersion of convex
hypersurfaces and, in particular, $M^n$ itself is simply connected and
diffeomorphic to $\Sph^k\times\Sph^{n-k}$.

If $\ell=k=1$, and hence $hol=\fso(n-1)$, \eqref{soul} implies that
$\tilde{M}^n$ splits isometrically $\tilde{M}^n=N^{n-1}\times \R$, with
$N^{n-1}$ compact and simply connected. But then
$hol(N^{n-1})=\fso(n-1)$, and again by \pref{rigidity} we get that
$N^{n-1}$ is diffeomorphic to $\Sph^{n-1}$. Thus
$M^n=(\Sph^{n-1}\times\R)/\Gamma$ isometrically, with $\Gamma$ a discrete
group acting by isometries. We will see in \cref{z} below that
$\Gamma=\Z$ and that $M^n$ is diffeomorphic to either
$\Sph^{n-1}\times\Sph^1$ if $M^n$ is orientable, or to the nonorientable
quotient $(\Sph^{n-1}\times\Sph^1)/\Delta\Z_2$ otherwise.

\bigskip

To conclude the proof of \tref{a}, it remains to study the possibilities
for the fundamental group when $\ell=k=0$ and $\tilde{M}^n=\Sph^n$, or
when $\ell=k=1$ and $\tilde{M}^n=\Sph^{n-1}\times\R$ for some metrics of
nonnegative curvature on the spheres. In order to do this, we first
recall the relationship, due to Kuiper, between height functions and
shape operators of compact Euclidean submanifolds.

Following \cite{m2}, for each $v\in\Sph^{n+1}$ consider the height
function $h_v=\ml v,\cdot\,\mr:M^n\to\R$ and define
$$
\mu_k(v)=\text{number of critical points of $h_v$ with index $k$},
$$
and
$$
\tau_k=\frac{1}{vol(\Sph^{n+1})}\int_{\Sph^{n+1}}\mu_k(v)dv.
$$
Since on a set of full measure the height functions are nondegenerate,
we can apply the usual Morse inequalities for Morse functions. After
integrating them we get
$$
\tau_k\ge b_k
$$
and
$$
\sum_{k=0}^{\ell}(-1)^{\ell-k}\,\tau_k\ge
\sum_{k=0}^{\ell}(-1)^{\ell-k}\, b_k,\ \ \ \forall\ \ell = 1,\dots n,
$$
where $b_k=\dim H_k(M,F)$ for some field $F$.

Clearly, $p$ is a critical point of $h_v$ if and only if $v=\beta$ for some
$(p,\beta)\in T_1^\perp M$, where $T_1^\perp M$ is the normal bundle of the
immersion. Moreover, $p$ is a nondegenerate critical point of $h_\beta$ if and
only if $A_\beta$ is nonsingular, and the index of the critical point is equal
to $ind(A_\beta)$, the number of negative eigenvalues of $A_\beta$.
We now use the Gauss map $G\colon T^\perp_1M\to \Sph^{n+1}$ with
$G(p,\beta)=\beta$, which for compact submanifolds is surjective. Since
$G_{*(p,\beta)}=\diag(A_\beta,\Id)$, we have that $p$ is a nondegenerate
critical point of $h_\beta$ if and only if $(p,\beta)$ is a regular point of
$G$. Thus if $C\subset
\Sph^{n+1}$ is the set of regular values of $G$ (an open and dense set of
full measure), then $h_v$ with $v\in C$ is a Morse function and hence
$\mu_k$ is constant on every connected component of $C$. Furthermore, for
$v\in C$, the set $G^{-1}(v)$ contains $\mu_k(v)$ many points of index $k$.
Thus by change of variables,
\begin{equation}\label{form}
\tau_k=\int_C\mu_k(v)dv=\frac{1}{vol(\Sph^{n+1})}
\int_{ind(A_\beta)=k }\abs{\det A_\beta}dvol_{T_1^\perp M}.
\end{equation}
Since $\mu_k(v)=\mu_{n-k}(-v)$, we also have $\tau_{n-k}=\tau_k$ for all
$k$.

\medskip

We proceed by applying \lref{l:eq} to the shape operators $A,B$
in \eqref{AB} in order to estimate $\abs{\det A_\beta}$ for
$(p,\beta)\in T_1^\perp M$, by first integrating over a circle in the
normal space $T_p^\perp M$, and then over $p\in M$. For a fixed normal
space $T_p^\perp M$, \lref{l:eq} implies that
\begin{equation}\label{detcomp}
|\det\left(\cos\theta A+\sin\theta B\right)|\geq|\det\left(\cos\theta
A-\sin\theta B\right)| \ \ \text{ for all }\ 0 \le \theta\le
\frac{\pi}{2}.
\end{equation}
Furthermore, if a regular value $\beta=\cos(\theta)\xi+\sin(\theta)\eta$
lies in the first quadrant, then \mbox{$ind(A_\beta)=0$}, and if it lies
in the third quadrant then $ind(A_\beta)=n$, whereas all saddle points
(and possibly some local maxima and minima as well) lie in the second and
fourth quadrants. Thus \eqref{detcomp}, together with
$\tau_{n-k}=\tau_k$, implies that
\begin{equation}\label{chen}
\tau_0+\tau_n\ge \tau_1 + \dots + \tau_{n-1}.
\end{equation}

As we will see, in the case of strict inequality $M$ is diffeomorphic to
a sphere. We are thus interested from now on in the equality case, where
one expects a certain amount of rigidity. This motivates our next key
definition.

\bigskip

{\it Definition.} Given a compact Riemannian manifold $M^n$ with
nonnegative sectional curvature, an isometric immersion
$f:M^n\to\R^{n+2}$ is said to be {\it \pepe} if
$\tau_0+\tau_n= \tau_1 + \dots + \tau_{n-1}$.

\bigskip

We can now use \lref{l:eq} to express the rigidity in terms of the shape
operators $A$ and $B$ in \eqref{AB}.

\begin{prop}\label{wide}
If $f \colon M^n\to \R^{n+2}$ is wide, then at each point $p\in M^n$ we
have that either $\ker A\oplus \ker B = T_p M$ with $A\neq0$ and
$B\neq0$, or $\ker A\cap\ker B \neq 0$.
\end{prop}
\begin{proof}
Equality in \eqref{chen} holds if and only if we have equality in
\eqref{detcomp} at every point and for every angle $\theta$. Thus
\lref{l:eq} implies that either $\ker A\oplus \ker B = T_p M$ or
$\ker A\cap\ker B \neq 0$. Furthermore, if $A=0$ at some point, then $B$
is nonsingular, and hence in a neighborhood of the point as well. But
then this neighborhood will not have any saddle points, and we would have
a strict inequality in \eqref{chen}.
\end{proof}

Notice that, along the (open) set where the first case holds, the
orthonormal basis $\{\xi,\eta\}$ in \eqref{AB}, and hence $A$ and $B$,
are unique and thus smooth. On the other hand, no such information can be
obtained in the second case.

\medskip

Altogether, we have the following  improvement of a result
in \cite{m2}, where we obtain additional information on the structure
of the immersion.

\begin{prop}\label{main}
If $M^n$ is compact, nonnegatively curved and immersed in Euclidean space
in codimension $2$, then we have for any coefficient field $F$ over which
$M^n$ is orientable that
$$
b_1+b_2+\dots  +b_{n-1}\le 2,
$$
with equality if and only if
$$
\tau_k=b_k \ \ \forall\ 2\leq k\leq n-2, \ \ \tau_1-\tau_0= b_1-1
\text{ and }\ \tau_0+\tau_n= \tau_1 + \dots + \tau_{n-1}.
$$
In particular, equality implies that the immersion is \pepe.
\end{prop}
\begin{proof}
By the Morse inequalities $\tau_k\ge b_k$ for $k=2,\dots n-2$ and
$\tau_1-\tau_0\ge b_1-b_0$, i.e. $b_1\le \tau_1-\tau_0+1$. If $M^n$ is
compact and orientable mod $F$, then $b_k=b_{n-k}$ and, in general,
$\tau_k=\tau_{n-k}$. Hence,
$$
b_1+(b_2+\dots +b_{n-2}) +b_{n-1}\le \tau_1-\tau_0+1 +(\tau_2+\dots
+\tau_{n-2}) + \tau_{n-1}-\tau_n+1,
$$
which together with \eqref{chen} yields
$$
b_1+b_2+\dots+b_{n-1}\le\tau_1\dots+\tau_{n-1}-(\tau_0+\tau_n)+2\le 2,
$$
with equality as claimed.
\end{proof}

So far we have shown that there are three situations: either $\tilde M^n
\cong \Sph^n$ diffeomorphically, or $\tilde M^n$ splits isometrically as
either $\tilde M^n\equiv\Sph^{n-1}\times\R$ or
$M^n\equiv\Sph^k\times\Sph^{n-k}$, for certain metrics of nonnegative
curvature on the spheres. As a corollary of \pref{main}, we conclude that
in the first case $M$ itself is a sphere when $n\geq 4$:

\begin{cor}\label{sn}
If $n\ge 4$ and $\pi_1(M^n)$ is finite, then $M^n$ is simply connected.
In particular, either $M^n$ is diffeomorphic to $\Sph^n$, or isometric
to a Riemannian product \mbox{$\Sph^k\times\Sph^{n-k}$}, $2\leq k\leq
n-2$, for certain metrics of nonnegative sectional curvature on the
spheres. In the latter case, the immersion is a product of two convex
hypersurfaces, and thus it is wide.
\end{cor}
\begin{proof}
Assume $M^n$ is not simply connected, and let $g\in\pi_1(M^n)$ be an
element of prime order $q$, so $\ml g\mr\cong\Z_q\subseteq \pi_1(M)$.
Then, the lens space $\Sph^n/\Z_q$ is a finite cover of $M^n$ and with
the covering metric it has an \ii into $\R^{n+2}$. But
$b_k(\Sph^n/\Z_q,\Z_q)=1$ for all $k$ which contradicts \pref{main} if
$n\geq 4$. Notice that if $q>2$ the lens space is orientable, and if
$q=2$ it is orientable mod~$2$.

In the case where $M^n$ is isometric to a Riemannian product
\mbox{$\Sph^k\times\Sph^{n-k}$}, \cite{am} implies that the immersion is
a product of two convex hypersurfaces. One easily sees that in this case
the immersion is wide.
\end{proof}

For the three dimensional case we have the following.

\begin{cor}\label{lens}
If $n=3$ and $\pi_1(M^3)\neq 0$ is finite, then $M^3$ is diffeomorphic to
a lens space $\Sph^3/\Z_k$, and the immersion is wide with
$\tau_0=\tau_1=\tau_2=\tau_3$. Furthermore, the height functions satisfy
$\mu_0=\mu_1$ as well as $\mu_{2}=\mu_3$.
\end{cor}
\begin{proof}
$M^3$ is again covered by a lens space with $b_1=b_2=1$ and we have
equality in \pref{main}, which implies that $\tau_0=\tau_1=\tau_2=\tau_3$
(they do not have to be $1$, though). The Morse inequality tells us that
$\mu_1(v)-\mu_0(v)\geq b_1-b_0=0$ for a.e.
$v\in\Sph^{n+1}\subset\R^{n+2}$, and since $\tau_1=\tau_0$, it follows
that $\mu_1(v)=\mu_0(v)$. From $\mu_k(v) = \mu_{n-k}(-v)$ we get
$\mu_3=\mu_2$ for a.e. $v$.
Therefore for almost all $v$ the Morse function $h_v$ has the same number
of critical points of index 0 and 1 (first for the lens space, and
then for $M^3$ as well). But if this is the case, the cancellation result
in \cite{mm} says that there exists another (abstract) Morse function on
$M^3$ with only one critical point of index 0 and 1. In
particular, this means that $M^3$ is a CW complex whose 1 skeleton is a
circle, which implies by transversality that the map
$\pi_1(S^1)\to \pi_1(M^3)$ induced by the inclusion
is onto and hence $\pi_1(M^3)$ is finite cyclic.
\end{proof}

For infinite fundamental group, we have:

\begin{cor}\label{z}
If $\pi_1(M^n)$ is infinite, then $\pi_1(M^n)\cong \Z$ is also cyclic.
Moreover, $\tilde M^n = \Sph^{n-1}\times\R$ splits isometrically for some
metric of nonnegative curvature in $\Sph^{n-1}$, and $M^n$ is
diffeomorphic to $\Sph^{n-1}\times\Sph^1$ if $M^n$ is orientable, or to
the nonorientable quotient $(\Sph^{n-1}\times\Sph^1)/\Delta\Z_2$
otherwise. In addition, $\tau_0=\tau_1=\tau_{n-1}=\tau_n$, and
$\tau_2=\cdots=\tau_{n-2}=0$, and hence the immersion is wide.
Furthermore, the height functions satisfy $\mu_0=\mu_1$ as well as
$\mu_{n-1}=\mu_n$, and $\mu_2= \cdots =\mu_{n-2}=0$.
\end{cor}
\proof
We already saw that under this assumption, $M^n$ is diffeomorphic to
$(\Sph^{n-1}\times\R)/\Gamma$ for some group $\Gamma$ acting properly
discontinuously. Since the quotient is compact, there exists a subgroup
$\Gamma'\simeq\Z$ which acts via translations on the $\R$ factor. Choose
a covering $f\colon M^*\to M$ such that the image of the fundamental
group under $f_*$ is equal to $\Gamma'$. Thus
$b_1(M^*,\Z_2)=b_{n-1}(M^*,\Z_2)=1$, and we can apply \pref{main} to the
induced immersion of $M^*$. As in the proof of \cref{lens}, it follows
that for the height functions on $M^*$ we have $\mu_1(v)=\mu_0(v)$ and
$\mu_{n-1}(v)=\mu_n(v)$ for a.e. $v$. The same thus holds for the height
functions of the immersion of $M^n$, and again as in the proof of
\cref{lens}, we see that $\pi_1(M)$ is cyclic, hence $\Gamma$ is
isomorphic to $\Z$. Finally, since the functions $\mu_i$ are nonnegative,
and $\tau_2= \cdots =\tau_{n-2}=0$, it follows that $\mu_2= \cdots
=\mu_{n-2}=0$ as well.

Now, projection onto the first factor gives rise to a fiber bundle
$M^n\to \Sph^1$ with fiber $\Sph^{n-1}$ which is hence isomorphic to
$\Sph^{n-1}\times[0,1]/(p,0)\sim (\sigma(p),1)$ for some
diffeomorphism~$\sigma$. If $\sigma$ is orientation preserving, $\sigma$
is homotopic to the identity and hence the bundle is trivial, in which
case $M^n$ is diffeomorphic to $\Sph^{n-1}\times\Sph^1$. If on the other
hand $\sigma$ is orientation reversing, $M^n$ is nonorientable and the
orientable double cover is diffeomorphic to $\Sph^{n-1}\times\Sph^1$ and
hence $M^n$ is diffeomorphic to $(\Sph^{n-1}\times\Sph^1)/\Delta\Z_2$.
Notice that, in both cases, the product structure is not necessarily
isometric.
\qed
\vspace{1.5ex}

The presence of points with positive curvatures imposes further
restrictions.

\begin{cor}\label{posric}
If there exists a point with positive Ricci curvature, then either $M^n$
is diffeomorphic to $\Sph^n$, or $M^n$ is isometric to a product
$\Sph^k\times\Sph^{n-k}$ and $f$ is a product of two convex
hypersurfaces, for some $2\leq k\leq n-2$. In particular, if there is a
point with positive sectional curvature, then $M^n$ is diffeomorphic to
$\Sph^n$.
\end{cor}
\begin{proof}
The fundamental group must be finite since otherwise \cref{z} implies
that $\tilde M^n$ splits off a real line. If $n=3$, then \pref{wide} implies
that at a point with $Ric>0$, one of $A,B$, say $A$, has rank 2, and
$B$ has rank one. But then $\hat{R}=\Lambda^2A+\Lambda^2B=\Lambda^2A$ and
hence all 2-planes containing $\ker A$ have curvature 0, thus
contradicting $Ric>0$. In the remaining case \cref{sn} proves our claim.
\end{proof}

In particular, Corollaries \ref{sn}, \ref{lens} and \ref{z} imply the
following.

\begin{cor}\label{mustbewide}
If $M^n$ is not diffeomorphic to a sphere, then the immersion is wide.
\end{cor}

\vskip 0.1cm
\section{Wide and locally wide immersions}
\medskip

In this section we obtain further information about the way the manifold
$M^n$ is immersed, that is, about the \ii $f$. As we saw in
\cref{mustbewide}, unless $M^n$ is diffeomorphic to the sphere, for which
the space of nonnegatively curved immersions is quite rich, the \ii $f$
must be \pepe. Hence the purpose of this section is to understand these
immersions. We have chosen this approach since there are interesting
wide immersions of spheres, as shown in Example 2.

\bigskip

We assume from now on that $M^n$ is a Riemannian manifold with nonnegative
sectional curvature, $n\geq 3$, and $f:M^n\to\R^{n+2}$ is an isometric
immersion. To provide a deeper understanding of the local phenomena, we
do not require  $M^n$ to be complete or compact unless otherwise
stated. Instead, we only assume that $f$ satisfies the local consequence
in \pref{wide} of being wide. In other words, again following the
notations in \eqref{AB}, we say that $f$ is {\it locally wide} if, at
every point $p\in M$, either
\begin{equation}\label{localpepe}
\ker A\cap\ker B \neq 0,
\ \ \ \text{or}
\ \ \ \ker A\oplus \ker B = T_p M\ \ \text{with}\ A,B\neq 0.
\end{equation}
 From now on we will also assume in the second case that
$\rk B\le \rk A$.

Recall that the {\it index of nullity} $\mu$ and the {\it nullity}
$\Gamma$ of $M$ are defined as
$$
\Gamma(p)=\{u\in T_pM
\colon R(a,b)u=0 \ \  \forall\ a,b\in T_pM\},  \  \text{ and }\
\mu(p)=\dim \Gamma(p).
$$
Furthermore, we have the {\it index of relative nullity} $\nu(p)$ of $f$
at $p$ defined as the dimension of the {\it relative nullity}
$\Delta(p)$ of $f$ at $p$,
$$
\Delta(p)= \{u\in T_pM\colon \alpha(u,v)=0 \ \, \forall \, v\in T_pM\}=
\ker A(p)\cap\ker B(p).
$$
Notice that $\mu$ is an intrinsic invariant, while $\nu$ is extrinsic. By
the Gauss equation, $\Delta\subseteq\Gamma$ and hence $\nu\le\mu$.
We will also use the well know fact (see e.g. \cite{chk,ma}) that the
nullity distribution, as well as the relative nullity distribution, is
integrable on any open set where it has constant dimension, and its
leaves are totally geodesic in $M^n$. In the case of the relative
nullity, the images of the leaves under $f$ are open subsets of affine
subspaces of the Euclidean space as well. In addition, if $M^n$ is
complete, the leaves of both distributions are complete on the open set
of minimal nullity or minimal relative nullity.

Choosing an orthonormal normal frame $\{\xi,\eta\}$ as in \eqref{AB}, by
\eqref{localpepe} $M^n$ can be written as the disjoint union
$$
M^n = K\cup U_1\cup \cdots\cup U_{[n/2]},
$$
where $K$ is the subset of positive index of relative nullity,
$$
K = \{p\in M^n: \nu(p) > 0\},
$$
and, for $1\leq k \leq [n/2]$, $U_k$ is  given by
$$
U_k = \{p\in M^n: \rk\, A(p) = n-k,\,\rk\, B(p) = k  \ \ \text{and} \ \ker A(p)\cap \ker B(p) = 0\}.
$$
 From the continuity of the eigenvalues and
$\hat{R}=\Lambda^2A+\Lambda^2B$ it follows that:

\begin{itemize}
\item $K$ is closed, and the sets $U_k$ are open with $\partial U_k \subset K$;
\item $U_2\cup\cdots\cup U_{[n/2]}$ is the set of all points with
positive Ricci curvature, $Ric>0$;
\item On $U_1$ we have $\hat{R}=\Lambda^2A$ and hence $\Gamma=\ker A  $ and $\mu=1$, $\nu=0$.
\end{itemize}

\begin{rem}\label{KU}
{\rm The above discussion and \cref{posric} imply that, if $M^n$ is
compact and $f$ is wide, then $f$ is isometric to a product of convex
hypersurfaces unless $U_k=\emptyset$ for all $k\geq 2$. Therefore,
for the immersions of types $(c)$ and $(d)$ in \tref{a} we have that
$M^n=K\cup U_1$, and, in particular, $\mu\geq 1$ everywhere.}
\end{rem}

In what follows, on $M^n\setminus K$ where we have seen that the frame
$\{\xi,\eta\}$ is unique and smooth, $w$ will denote the normal
connection form of $f$ given by
$$
w(X):=\la\nabla^\perp_X\xi,\eta\ra
$$
for $X\in TM^n$, and thus $\nabla^\perp_X\xi=w(X)\eta$ and
$\nabla^\perp_X\eta=-w(X)\xi$. Hence the Codazzi equations $(\nabla_X
A)(Y,\beta)=(\nabla_Y A)(X,\beta)$ become
$$\nabla_X AY-A\nabla_XY-w(X)BY = \nabla_Y AX-A\nabla_YX-w(Y)BX,$$
$$\nabla_X BY-B\nabla_XY+w(X)AY = \nabla_Y BX-B\nabla_YX + w(Y)AX,$$
or, equivalently,
$$
\nabla_X AY-\nabla_Y AX=A([X,Y])+B(w(X)Y-w(Y)X)
$$
$$
\nabla_X BY-\nabla_Y BX=B([X,Y])-A(w(X)Y-w(Y)X),
$$
while the Ricci equation is
$$\ml R^\perp(X,Y)\xi,\eta\mr = dw(X,Y) = \la[A,B]X,Y\ra.$$

\smallskip

For the sake of completeness, let us first analyze the local behaviour of
$f$ on $U_k$ for $k\geq 2$. We will see that it is already a product of
strictly convex hypersurfaces, hence showing that case $(b)$ in \tref{a}
has its roots in a local phenomenon.

\begin{prop}\label{prod}
Assume that $f$ is locally wide and $2\leq k \leq [n/2]$. Then the
immersion $f|_{U_k}$ is locally a product immersion of two strictly
convex Euclidean hypersurfaces whose respective normal vectors are $\xi$
and $\eta$.
\end{prop}
\proof
Since their dimensions are constant, the distributions $\ker A$ and $\ker
B$ are smooth, have rank bigger or equal than two, and satisfy
$\ker A\oplus \ker B = TU_k$.

For $X,Y\in \ker A$, the Codazzi equations imply that $A[X,Y] =
B(w(X)Y-w(Y)X)$. Since $\im A \cap \im B=0$, both sides have to be $0$
and hence $\ker A$ is integrable. Furthermore, $\ker A \subset \ker w$
since $\ker A \cap \ker B = 0$ and $\dim\ker A\geq 2$. Indeed, If $AX=0$,
choose a linearly independent $Y\in \ker A \cap\ker\omega$, which implies
$w(X)BY=w(Y)BX$ and hence $\omega(X)=0$. Analogously, $\ker B$ is
integrable and $\ker B \subset \ker w$. Therefore, $w|_{U_k}=0$ which
implies that $f|_{U_k}$ has flat normal bundle. From the Ricci equation
it now follows that $A$ and $B$ commute, and hence $\ker A \perp \ker B$.

For $X,Y\in \ker A= \im B $ and $U\in \ker B=\im A $ the Codazzi
equations imply that $A[X,U] = \nabla_X AU$. Since $U$ is arbitrary,
$\nabla_XU\in \im A$ as well. Hence
$\ml \nabla_X Y,U\mr=-\ml Y,\nabla_X U\mr=0$ which implies that the
distribution $\ker A$ is totally geodesic, and similarly so is $\ker B$.
We conclude that both $\ker A$ and $\ker B$ are mutually orthogonal
transversal totally geodesic distributions, and hence both are parallel,
and $M^n$ is locally a product. Furthermore,
$\alpha_f(X,U)=\ml AX,U\mr\xi+\ml BX,U\mr\eta=0$ for all $X\in \ker A$
and $U\in\ker B$. The proposition then follows from the Main Lemma in
\cite{m1}. Observe that each factor is strictly convex since $A$ and $B$
are positive definite on the corresponding factors.
\qed
\vspace{1.5ex}

Before continuing we point out the following consequence.

\begin{cor}\label{ric}
If $M^n$ is complete, $Ric_M>0$, and $f$ is locally wide,
then $f$ is a global product of two strictly convex embedded
hypersurfaces.
In particular, $f$ is rigid.
\end{cor}
\proof
Consider $\pi:\tilde M^n\to M^n$ the universal cover of $M^n$
with the covering metric, and $\tilde f = f \circ \pi$. By hypothesis,
$K$ and $U_1$ are empty. Since $\partial U_k\subset K$ for all
$k$, there is $k_0\geq 2$ such that $U_{k_0}=\tilde M^n$. But then
\pref{prod} and the deRham decomposition theorem imply that $\tilde M^n$
is globally a product. From \cite{am} it follows that $\tilde f$
is a product immersion of two strictly convex Euclidean hypersurfaces.
By \cite{sa1}, a complete strictly convex Euclidean hypersurface
is the boundary of a strictly convex body and hence embedded. Thus
$\tilde f$ is injective and so $\tilde f = f$ and $\tilde M^n=M^n$.
\qed
\vspace{1.5ex}

The description of $f$ on the set $U_1$ is considerably more delicate.
This case is of main interest to us since, as pointed out in \rref{KU},
for the immersion in case $(c)$ and $(d)$ of \tref{a}, we have
$M=K\cup U_1$. Furthermore, $U_1$ is nonempty since otherwise the open
subset of minimal relative nullity in $K$ would be foliated by complete
straight lines in Euclidean space, contradicting compactness.

\vspace{1.5ex}

We need the following definition from \cite{df1}.

\begin{definition}\label{comp} {\rm
Given an \ii $g\colon\,M^n\to\R^{n+1}$, we say that another \ii
$f\colon\,M^n\to\R^{n+2}$ is a {\it composition (of $g$)}
when there is an isometric embedding $g'\colon M^n\hookrightarrow
N_0^{n+1}$ into a flat manifold $N_0^{n+1}$, an \ii
$j\colon\,N_0^{n+1}\to\R^{n+1}$ (that is, a local isometry) satisfying
$g=j\circ g'$, and an \ii $h\colon\,N_0^{n+1}\to\R^{n+2}$ such that
$f=h\circ g'$.}
\end{definition}

Observe that, for any open subset $U\subset M^n$ where $g_{|U}$ in the
above is an embedding, we can assume that $N_0^{n+1}$ is an open subset
on $\R^{n+1}$ and $j$ is the inclusion.

Compare the following local description of $f$ on $U_1$
with the structure of Example 1 in the introduction.

\begin{prop}\label{b0} If $\pi:\tilde U_1\to U_1$ is the universal cover
of $U_1$, then $f\circ\pi$ is a composition of a convex
Euclidean hypersurface with constant index of relative nullity one.
\end{prop}
\proof
Since $\rk B=1$, and hence $\hat{R}=\Lambda^2A$, the shape operator $A$
alone satisfies the Gauss equation along $U_1\subset M^n$. We claim that
$A$ also satisfies the Codazzi equation for hypersurfaces, that is, the
skew symmetric tensor $S(X,Y):=\nabla_XAY - \nabla_YAX - A[X,Y]$
vanishes.

For any $X,Y\in TU_1$, the Codazzi equation for $A$ tells us that $S(X,Y)
= B(w(X)Y-w(Y)X)$. For $X,Y\in\ker B$, the Codazzi equation for $B$ says
that $B[X,Y]=-A(w(X)Y-w(Y)X)$. Since $\dim \ker B\ge 2$, it follows as in
the proof of \pref{prod} that $\ker B \subset \ker w$. Thus either
$\ker B = \ker w$ or $w=0$. In the first case, since $S$ is skew
symmetric, we can assume that $X$ and $Y$ are linearly independent, and
hence can assume that $X\in\ker\omega$. But then $S(X,Y)=0$. Altogether,
$S$ vanishes.

By the Fundamental Theorem of Submanifolds, locally on $U_1$ (or globally
on $U_1$ if it is simply connected), there exists a Euclidean
hypersurface $g$ whose second fundamental form is $A$.
Since $\rank A = n-1 \geq 2$, $\nu_g=\mu|_{U_1}\equiv 1$.
Now, since $\ker B \subset \ker w$, by Proposition~8 in \cite{df1} we
have that $f$ is a composition of $g$ (see also Proposition~9 in
\cite{df}).
\qed

\begin{rem*}
Since $\nu_g=1$, the nullity geodesics in $U_1$ are (locally) mapped by
$g$ into straight lines of $\R^{n+1}$. If $M=U_1$ is complete and simply
connected, we will see in the proof of \tref{b1} that these straight
lines are parallel and $g$ is globally a cylinder.
\end{rem*}

\vspace{1ex}

We will also discuss properties of the metric and the immersion on the
set of index of nullity $1$, i.e., on $V:=\mu^{-1}(1)$, and the open set
$U_1'$ where $B$ does not vanish (and hence $f$ is 1-regular, i.e.,
$\dim\spa \im(\alpha) = 2$),
$$
U_1'= \{p\in V: \rank B(p) = 1\}\supseteq U_1.
$$
Recall that by
the Gauss equations on the set $U_1$ we have $\mu=1$ and, by definition,
$\nu=0$ as well. Thus,
$$
U_1\subseteq U_1'\subseteq V\subseteq K\cup U_1 \subseteq M^n.
$$
On the complement of $U_1$ in $V$ we have $\nu=\mu=1$.
In any case, $V$ is the set of minimal index of nullity of $M^n$,
so its leaves of nullity are complete if $M^n$ is complete.

\medskip

The global version of \pref{b0} is the following. This in particular
applies to any immersion in cases $(c)$ and $(d)$ of \tref{a}.

\begin{main}\label{b1}
Assume that $M^n$ is compact with nonnegative sectional curvature,
and $f$ is locally wide. Furthermore, assume
that $M^n$ has no points with positive Ricci curvature, i.e.,
$M=U_1\cup K$. Let $\pi\colon\tilde V \to V$ the universal cover of $V$,
$\tilde f=f\circ\pi$ the lift of $f|_V$, and set
$\tilde{U}_1'=\pi^{-1}(U_1')\subset \tilde V$. Then we have:
\begin{itemize}
\item[$i)$] $\tilde V$ splits globally
and isometrically as a product $\tilde V = N^{n-1} \times \R$, where
$g\colon N^{n-1}\to \R^n$ is a strictly locally convex hypersurface.
In particular, $\tilde V$ is itself an Euclidean hypersurface via the
cylinder over $g$, $g\times\Id_\R:\tilde V\to\R^{n+1}$;
\item[$ii)$] The restriction $\tilde f|_{\tilde{U}_1'}$ is a
composition of the cylinder over $g|_{\tilde{U}_1'}$;
\item[$iii)$] Along each connected component $W$ of the interior of
$\tilde V\setminus \tilde U_1'$, $\tilde f|_W$ is a composition of
the cylinder over $g|_W$ with a linear inclusion
$\R^{n+1}\subset\R^{n+2}$.
\end{itemize}
\end{main}
\proof
Here we continue with the notations in the proof of \pref{b0}. Define on
$V$ the {\it splitting tensor} of the (totally geodesic and complete)
nullity distribution, $C:\Gamma^\perp \to \Gamma^\perp$ given by
$CX=-\nabla_XT$, where $T$ is a unit vector field tangent to $\Gamma$.
Since $\Gamma$ is totally geodesic, the distribution $\Gamma^\perp$ is
totally geodesic if and only if $C$ vanishes.

Observe that $C$ satisfies the Riccati type differential equation
$C'=C^2$ when restricted to a geodesic with $\gamma' = T\in\Gamma$.
Indeed, $C'X = \nabla_TCX - C \nabla_TX =  -\nabla_T\nabla_XT +
\nabla_{\nabla_TX}T = -\nabla_X\nabla_TT - \nabla_{[T,X]}T -R(T,X)T+
\nabla_{\nabla_TX}T = \nabla_{\nabla_XT}T=C^2X$ since $\nabla_TT=0$ and
$R(T,X)T=0$. Since $\gamma$ is complete, this ODE holds true over the
entire real line, and then $C(\gamma(t))=C_0(I-tC_0)^{-1}$, where
$C_0=C(\gamma(0))$. Therefore, along each nullity geodesic, all real
eigenvalues of $C$ vanish.

We claim that $C\equiv 0$. Observe first that every nullity line in $V$
has to intersect $U_1$ since in $V\setminus U_1$ we have $\nu=1$ and
hence a complete nullity line is also a relative nullity line. But then
its image under $f$ is a straight line in $\R^{n+2}$, contradicting
compactness of $M^n$. As we saw in the proof of \pref{b0}, on $U_1$
the shape operator $A$
satisfies the Codazzi equation of a hypersurface and hence
$\nabla_TAX=\nabla_XAT+A[T,X]=A[T,X]$ since on $U_1$ we also have $AT=0$.
Thus if $X\in TU_1$ and
$X\in\Gamma^\perp$ we have
\begin{equation}\label{a'}
A'X = \nabla_{T} AX - A\nabla_{T}X =A[T,X]-A\nabla_{T}X
= -A\nabla_XT = ACX.
\end{equation}
Hence, from the symmetry of $A'$, we have $A C=(A C)^t=C^tA $.
Furthermore, since $A>0$ on $\Gamma^\perp$, we have the inner product
$\la X,Y\ra_1=\la AX,Y\ra$, positive definite on $\Gamma^\perp$. But then
$C$ is self adjoint since
$\la X,CY\ra_1=\la AX,CY\ra=\la C^tAX,Y\ra=\la ACX,Y\ra=\la CX,Y\ra_1$.
Thus all eigenvalues of $C$ are real, and hence $C\equiv 0$ on $V$
which proves the claim.

In particular, $V$ is locally a Riemannian product of a line and a
manifold with positive sectional curvature. Moreover, by \eqref{a'}, $A$
is parallel along $\gamma$ in $U_1$, which implies in particular that
$\xi$ and $\eta$ are parallel along $U_1$ as well.

Now, in $V\setminus U_1$, $T$ spans the relative nullity of $f$. For any
parallel normal vector field $\sigma$ along $\gamma$ we obtain as in
\eqref{a'} that $A_\sigma'=A_\sigma C = 0$.     In particular, by taking the
parallel transport of $\xi$ we see that, along the whole $V$, there is a
unique smooth unit normal vector field $\xi$ such that $A=A_\xi\geq 0$
has rank equal to $n-1$, and for a local unit normal vector field
$\eta\perp\xi$, $B=A_\eta$ has rank at most 1. In addition, $A$ is
parallel along the complete lines of nullity in $V$, and $w(T)=0$ along
$V\setminus U_1$.

We now show that $\tilde V$ is a cylinder over a strictly convex
Euclidean hypersurface, thus proving part $(i)$.
First, the Ricci equation along $V\setminus U_1$ tells us that
$$
0=R^\perp(X,T)=T(w(X))-X(w(T))-w(\nabla_TX-\nabla_XT)
=T(w(X))-w(\nabla_TX),
$$
since $\Gamma=\Delta$ in $V\setminus U_1$.
Therefore, $\ker w$ is parallel along $\gamma$. Moreover, since $B'=BC=0$,
$\ker B$ is also parallel along $\gamma$. But in the proof of
\pref{b0} we showed that, on $U_1$, $\ker B\subset\ker w$.
We conclude that $\ker B\subset\ker w$ on the whole $V$, and, again
as in the proof of \pref{b0}, $A$ satisfies Gauss and Codazzi equations
for $V$. So, the tensor $\tilde A=\pi^*A$ satisfies Gauss and Codazzi
equations for $\tilde V$, and thus there exists an isometric immersion
$g':\tilde V\to\R^{n+1}$ with gauss map $\tilde\xi$ whose second
fundamental form is $\tilde A$. By construction, $\Gamma$ is now the
relative nullity of $g'$, and hence the complete geodesics of nullity are
mapped into complete straight affine lines of $\R^{n+1}$. On the other
hand, since
$$
\tilde \nabla_X g'_*T = \la\tilde AT,X\ra\tilde\xi = 0,
\ \ \ \forall X\in TV,
$$
for the standard connection $\tilde \nabla$ of $\R^{n+1}$, the complete
leaves of $\Gamma$ are mapped by $g'$ into locally parallel lines of
$\R^{n+1}$. Therefore, $\tilde V$ is globally a product $\tilde V =
N^{n-1}\times \R$ and $g'$ is a cylinder, i.e., $g'=g\circ\Id_\R$, where
$g:N^{n-1}\to \R^n$ is a strictly convex hypersurface, as claimed.

Finally, since $\ker B \subset \ker w$, part $(ii)$ follows from
Proposition~8 in \cite{df1} as in \pref{b0}, while part $(iii)$ is
immediate from $B=0$ and $w=0$ (i.e. $\eta$ is locally constant) on
$V\setminus U_1$.
\qed

\begin{rem}\label{1reg}
{\rm The map $h$ in \dref{comp} may fail to be an immersion along the
image of $j$ on the boundary points of $\tilde U_1'$.}
\end{rem}

The above easily implies the following corollary. Note that this applies
in particular to the immersions in part $(c)$ of \tref{a} if the metric
on $\Sph^{n-1}$ has positive sectional curvature.

\begin{cor}\label{2pos} Under the assumptions of \tref{b1}, suppose
further that $M^n$ has 2-positive Ricci curvature. Then, for the
universal cover $\pi:\tilde M^n\to M^n$ we have
$\tilde M^n=\Sph^{n-1}\times\R$, where $g:\Sph^{n-1}\to\R^n$ is a compact
strictly convex embedded hypersurface. Moreover, if $f$ is 1-regular,
then $f\circ\pi$ is globally a composition of its cylinder:
$f\circ\pi=h\circ (g\times \Id_\R)$.
\end{cor}
\proof
The hypothesis on the Ricci curvature is equivalent to $\rk A=n-1$
everywhere, and hence $V=M^n$. The 1-regularity of $f$ implies that
$U_1'=V=M^n$. The result then follows from \tref{b1}.
\qed
\vspace{1.5ex}

For $n=3$, if $M^3$ is compact but not diffeomorphic to a
sphere, we have the decomposition
$$
M^3=K\cup U_1=F\cup V,
$$
where $F$ is the set of flat points of $M^3$. Indeed, in dimension 3,
$\mu(p)>1$ implies that $p\in F$.

\smallskip

Compare our next result with Examples 1 and 2 in the
introduction.

\begin{cor}\label{3}
Let $f:M^3 \to \R^5$ be an \ii of a compact manifold
with nonnegative sectional curvature. If $M^3$ has no complete geodesic
of flat points, then $M^3$ is either diffeomorphic to $\Sph^3$, or its
universal cover $\tilde M^3$ is isometric to $\Sph^2\times\R$ for some
metric of positive Gaussian curvature on the sphere $\Sph^2$.
\end{cor}
\proof
By \cref{mustbewide}, $f$ is wide, hence locally wide. As we saw in the
proof of \tref{b1}, $V$ is then foliated by complete geodesics of
nullity, and therefore its boundary as well. But since $M^3=F\cup V$,
the boundary of $V$ is made of flat points. Therefore,
$M^3$ has no flat points, and $\mu=1$ everywhere. By part $(i)$ of
\tref{b1}, $\tilde M^3=\tilde V=N^2\times\R$ splits globally and
isometrically. That $N^2$ is a sphere is a consequence of \tref{a}.
\qed

\vspace{1.5ex}

Part $(i)$ of \tref{b1} also immediately implies the next corollary, which has
\tref{3a} as a consequence.

\begin{cor}\label{lens1}
Let $f:M^3=\Sph^3/\Z_k \to \R^5$ with $k>1$ be an isometric immersion with
nonnegative sectional curvature. Then $M^3$ has flat points, and the
complement $V$ of its flat points is isometric to a twisted cylinder
$(N^2\times\R)/\Z$, where $N^2 \subset \R^3$ is a surface with positive
Gaussian curvature.
\end{cor}

We can now use the above and \cref{lens} to compute the type numbers
$\tau_k$ for case $(d)$ in \tref{a}.

\begin{cor}\label{taus}
In the situation of \cref{lens1} and hence \tref{a} part $(d)$, the type
numbers are given by $\tau_k=\frac{1}{8\pi^2}\int_{N^2}K(x)\kappa_g(x)dx$ for
all $0\leq k \leq 3$. Here, $K$ denotes the Gaussian curvature of $N^2$ and
$\kappa_g(x)$ the total curvature of the leaf of $\Gamma$ through $x\in N^2$,
considered as a curve in $\R^5$.
\end{cor}
\proof
Notice first that the computation of $\tau_k$ in \eqref{form} only takes
into account the normal bundle over the set with vanishing relative
nullity, which in our case is $U_1\subset V$. Writing
$\beta=\cos(\theta)\xi+\sin(\theta)\eta$, a direct computation using
$\rk A\leq 2$ and $\rk B\leq 1$ shows
that $\det A_\beta=\cos(\theta)^2\sin(\theta)\det(\hat A)\la BZ,Z\ra$ on
$V$, where $\hat A$ is the restriction of $A$ to $\Gamma^\perp=TN^2$
and $Z=\gamma_x'\in\Gamma=\ker A$ for $\gamma_x(t)=[(x,t)]$.
Along $V$ the surface $N^2$ in \cref{lens1} has $\det(\hat A)=K\geq 0$.
Furthermore, $|\la BZ,Z\ra|= \|\tilde\gamma_x''\|$ is the geodesic
curvature of $\tilde\gamma_x:=f\circ\gamma_x$ in $\R^5$ since $AZ=0$ and
hence $\tilde\gamma_x''=\alpha (\gamma',\gamma')=\la BZ,Z\ra\eta$. Thus
\eqref{form} implies that
$$
\tau_k = \frac{1}{4}\sum_{i=0}^3\tau_i=
\frac{3}{32\pi^2}\int_{T_1^\perp\!V}\abs{\det A_\beta}=
\frac{1}{8\pi^2}\int_VK(x)\|\tilde\gamma_x''(t)\|=
\frac{1}{8\pi^2}\int_{N^2}K(x)\kappa_g(x)dx.
$$
Here, $\kappa_g(x)=\int_0^a\|\tilde\gamma_x''(t)\|dt$ is the total
geodesic curvature of $\tilde\gamma_x\subset\R^5$ until the return time
$a>0$, i.e., the cyclic group $\Z$ in \cref{lens1} is spanned by
$(x,t)\mapsto(j(x),t+a)\in\Iso(N\times\R)$.
\qed
\vspace{1.5ex}

In particular, for the switched 3-sphere $M^3_\e$ in Example 2
we get $\tau_k=1$, and hence the immersion is tight.

\end{document}